\documentclass{article}

\usepackage{amssymb, latexsym,pdfsync,amsmath,amsthm,color}
\newtheorem{theorem}{Theorem}
\newtheorem{corollary}{Corollary}
\newtheorem{lemma}{Lemma}

\theoremstyle{definition}
\newtheorem{definition}{Definition}
\newtheorem{remark}{Remark}

\newcommand{\SSS}{\mathbb{S}}

\newcommand{\FF}{\mathbb{F}}

\newcommand{\Fq}{\mathbb{F}_q}
\newcommand{\Fqn}{\mathbb{F}_{q^n}}

\newcommand{\D}{\mathcal D}
\newcommand{\cS}{\mathcal S}
\newcommand{\cF}{\mathcal F}

\def\S{\mathbb{S}}

\def\F{\mathbb{F}}
\def\Fq{{\mathbb{F}}_q}

\def\End{\mathrm{End}}
\def\Aut{\mathrm{Aut}}

\def\PG{\mathrm{PG}}
\def\PGL{\mathrm{PGL}}
\def\GL{\mathrm{GL}}

\def\GammaL{\mathrm{\Gamma L}}
\def\dim{\mathrm{dim}}

\def\rk{\mathrm{rk}}
\def\trk{\mathrm{trk}}

\def\Tr{\mathrm{Tr}}

\def\mrk{\mathrm{mrk}}
\def\brk{\mathrm{brk}}
\def\GTF{\mathrm{GTF}}

\newcommand{\npmatrix}[1]{\left( \begin{matrix} #1 \end{matrix} \right)}

\newcommand{\rank}{\mathrm{rank}}

\begin{document}
\title{The BEL-rank of finite semifields}
\author{Michel Lavrauw and John Sheekey}
\date{\today}
\maketitle

\begin{abstract}
In this article we introduce the notion of the {\it BEL-rank} of a finite semifield, prove that it is an invariant for the isotopism classes, and give geometric and algebraic interpretations of this new invariant. Moreover, we describe an efficient method for calculating the BEL-rank, and present computational results for all known small semifields.
\end{abstract}



\section{Introduction}

A (finite) semifield is a division algebra in which multiplication is not assumed to be associative. For background and definitions, we refer to \cite{Kantor2006}, \cite{LaPo2011}.

We write the product of two elements $x,y$ in a semifield $\S$ by $\S(x,y)$. Let $\SSS$ be an $n$-dimensional algebra over $\Fq$, i.e. an algebra of order $q^n$ with centre containing $\Fq$. We identify the elements of $\SSS$ with the elements of $\Fqn$. Then there exist unique elements $c_{ij} \in \Fqn$ such that
\[
\SSS(x,y) = \sum_{i,j=0}^{n-1} c_{ij} x^{q^i}y^{q^j}.
\]

Two semifields $\SSS$ and $\SSS'$ are {\it isotopic} if there exist nonsingular linear maps $F,G,H:\SSS \rightarrow \SSS'$ such that
\[
\SSS'(F(x),G(y)) = H(\SSS(x,y))
\]
for all $x,y \in \SSS$. We denote the {\it isotopy class} of $\SSS$ by $[\SSS]$.

Two semifields are {\it strongly isotopic} if they are isotopic via the triple $(F,G,I)$, where $I$ denotes the identity map. Note that this definition varies in the literature, with one of $F$ or $G$ sometimes being take to be the identity instead of $H$.

Given a semifield $\S$, we can define vector spaces
\begin{align*}
S_y &:= \{ (x,\S(x,y)):x \in \Fqn\} \leq \Fqn \times \Fqn\\
S_{\infty} &:= \{ (0,x):x \in \Fqn\} \leq \Fqn \times \Fqn.
\end{align*}
Then the set $\cS_q(\S) := \{S_y : y \in \Fqn\} \cup \{S_{\infty}\}$ is a {\it semifield spread} of $\Fqn \times \Fqn$, which we identify with $V(2n,q)$. Note that this definition depends on the choice of a field $\Fq$ contained in the centre of $\S$. We will write $\cS(\S)$ when the ambient vector space is clear. Throughout wherever we consider a ``semifield of order $q^n$'', we will implicitly assume that $\cS(\S) = \cS_q(\S)$. See e.g. \cite{LaPo2011} for details on semifield spreads. Two semifield spreads $\cS$ and $\cS'$ are said to be equivalent (written $\cS \simeq \cS'$) if there exists some $\phi \in \GammaL(2n,q)$ such that $\cS' = \cS^{\phi}$. It is well-known that $\cS(\S)$ is equivalent to $\cS(\S')$ if and only if $\S$ and $\S'$ are isotopic.

For each semifield we can define the {\it dual} $\S^d$ and {\it transpose} $\S^t$, which together give a set of up to six isotopy classes of semifields known as the {\it Knuth orbit}, as defined in \cite{BaLa2007}. Similarly, for each semifield spread we can define the {\it dual spread} $\cS^d$ and {\it transpose spread} $\cS^t$, giving up to six equivalence classes of spreads.

A new construction for semifield spreads was introduced in \cite{BEL2007}, using {\it BEL-configurations} (see Section \ref{sec:belconfig} for definitions). In \cite{BEL2007} it was shown that every semifield spread in $V(2n,q)$ can be constructed from a BEL-configuration in $V(rn,q)$ for some $r \leq n$. This suggests the question:
\begin{center}
What is the smallest $r$ such that $\cS$ can be obtained from a BEL-configuration in $V(rn,q)$?
\end{center}
We will address this in Section \ref{sec:belconfig} by defining the {\it BEL-rank} of a semifield, and give an interpretation of this in terms of the multiplication of an associated semifield. In Section \ref{sec:spreadsets}, we will give another interpretation in terms of spread sets and Desarguesian spreads. In Sections \ref{sec:mrk} and \ref{sec:bmrk} we will exhibit a general method for calculating this invariant, plus some specific examples. In Section \ref{sec:trk} we will compare this with the tensor rank of a semifield. In Section \ref{sec:comp} we will present some computational results, and some open questions arising from them.

\section{BEL-configurations and BEL-rank}
\label{sec:belconfig}

\begin{definition}
\label{def:BEL}
A \emph{BEL-configuration} in $V(rn,q)$ is a triple $(\D,U,W)$ such that:
\begin{itemize}
\item
$\D$ is a Desarguesian $n$-spread;
\item
$U$ is a subspace of dimension $n$;
\item
$W$ is a subspace of dimension $rn-n$;
\item
no element of $\D$ intersects both $U$ and $W$ nontrivially.
\end{itemize}
\end{definition}

In \cite{BEL2007}, a semifield spread of $\frac{V(rn+n,q)}{W} \simeq V(2n,q)$ was defined for each $(\D,U,W)$ a BEL-configuration in $V(rn,q)$ (which we will denote by $\cS(\D,U,W)$), and the following proved. 

\begin{theorem}{\rm \cite[Theorem 4.1]{BEL2007}}
Every semifield spread in $V(2n,q)$ can be constructed from a BEL-configuration in $V(rn,q)$ for some $r \leq n$.
\end{theorem}

This suggests the following definition.
\begin{definition}
\label{def:brk}
The {\em BEL-rank} of a semifield spread $\cS$ in $V(2n,q)$ is defined to be the minimum $r$ such that there exists a BEL-configuration in $V(rn,q)$ with $\cS(\D,U,W) \simeq \cS$. We will denote the BEL-rank of $\cS$ as $\brk(\cS)$. 
\end{definition}
\begin{definition}
\label{def:brk2}
The  {\em BEL-rank} of a semifield $\S$ is defined to be the BEL-rank of the semifield spread $\cS_q(\S)$, where $\Fq$ is the centre of $\S$. We will denote the BEL-rank of $\S$ as $\brk(\S)$. 
\end{definition}

Though a semifield $\S$ with centre $\FF_q$ defines a spread $\cS_{q'}(\S)$, for each $\FF_{q'}\leq \Fq$, we show now that the BEL-rank of each of these spreads is the same, and equals $\brk(\S)$. In the proof, we make use of the {\em field-reduction map} defined in \cite{LaVa2015}: the map $\mathcal{F}_{n,t,q}$, which maps subspaces of $V(n,q^t)$ to subspaces of $V(nt,q)$.

\begin{theorem}
\label{thm:nodep}
Let $\S$ be a semifield, let $\Fq$ be the centre of $\S$, and $\FF_{q'}$ a subfield of $\Fq$. Then $\brk(\cS_q(\S)) = \brk(\cS_{q'}(\S))$.
\end{theorem}

\begin{proof}
Suppose $\FF_{q'}$ is a subfield of $\Fq$ with $q = q'^t$. Suppose we have a BEL-configuration $(\D,U,W)$ in $V(rn,q)$  such that $\cS(\D,U,W) \simeq \cS_{q}(\S)$. Applying the field-reduction map $\mathcal{F}_{rn,t,q'}$, we obtain a BEL-configuration $(\D',U',W')$ in $V(r(nt),q')$, such that $\cS(\D',U',W') \simeq \cS_{q'}(\S)$. Hence $\brk(\cS_q(\S)) \leq \brk(\cS_{q'}(\S))$. 

Conversely, suppose we have a BEL-configuration $(\D',U',W')$ in $V(r(nt),q')$. In \cite{BEL2007} it was shown that $\cS(\D',U',W')$ is equivalent to a spread containing $U'$. The centre $\Fq = \FF_{q'^t}$ is a subfield of the kernel of $\S$, and hence a subfield of the group of endomorphisms fixing the spread $\cS_{q'}(\S)$ \cite{Andre1954}. This implies that there exists a Desarguesian $t$-spread $\D_t$ in $V(rnt,q')$ such that $U'$ is partitioned and spanned by elements of $\D_t$. Hence $U'$ is the image under the field-reduction map $\mathcal{F}_{rn,t,q'}$ of some subspace $U$ of $V(rn,q)$ which implies the existence of a BEL-configuration $(\D,U,W)$ in $V(rn,q)$ such that $\cS(\D,U,W) \simeq \cS_{q}(\S)$. Therefore $\brk(\cS_{q'}(\S)) \leq \brk(\cS_{q}(\S))$, and hence $\brk(\cS_q(\S)) = \brk(\cS_{q'}(\S))$, completing the proof.
\end{proof}

Note that if $r=1$, the only possibility is $U = V(n,q)$, $W = \{0\}$, and then $\S(\D,U,W) = \D$ is Desarguesian. Hence the only semifields of BEL-rank one are finite fields. It was shown in \cite{BEL2007} that generalised twisted fields and rank two semifields have BEL-rank at most (and hence equal to) two, and the Coulter-Matthews semifield has BEL-rank at most three.

In \cite{LaSh2014} the following was shown.
\begin{theorem}
\label{thm:Sfg}{\rm \cite[Theorem 3]{LaSh2014}}
Suppose $(\D,U,W)$ a BEL-configuration in $V(rn,q)$. Then there exist $2r$ $\Fq$-linear maps $f_1,\ldots,f_r$, $g_1,\ldots,g_r$ from $\Fqn$ to itself such that 
\begin{equation}
\label{eqn:Sfg}
\S(x,y) = \sum_{i=1}^r g_i(f_i(x)y)
\end{equation}
defines a semifield multiplication, and $\cS(\S) \simeq \cS(\D,U,W)$.

Conversely, if $\S$ is a semifield with multiplication as in (\ref{eqn:Sfg}), then there exists a BEL-configuration $(\D,U,W)$ in $V(rn,q)$ such that $\cS(\D,U,W) \simeq \cS(\S)$.
\end{theorem}

\begin{remark}
In \cite{Lavrauw2011}, it was shown that we may replace $U$ with a subspace over a subfield $\FF_{q_0}$ of the centre of the appropriate dimension (projectively, a {\it linear set}). The above theorem is still valid for such spaces, with the only difference being that the $f_i$'s are now $\FF_{q_0}$-linear maps.
\end{remark}

In \cite[Theorem 7]{LaSh2014} it was proved that $\cS(\D,W^{\perp},U^{\perp})$ is equivalent to $\cS(\D,U,W)^t$, where $\perp$ is a certain polarity on $V(rn,q)$. The following is an immediate consequence of this result.
\begin{theorem}
\label{thm:transpose}
For any semifield spread $\cS$, $\brk(\cS^t) = \brk(\cS)$.
\end{theorem}

\begin{remark}
\label{rem:notknuthinv}
The BEL-rank is {\it not} a Knuth orbit invariant, as we will see in Section \ref{sec:comp}.
\end{remark}

We note the following for future reference, which was stated for some special cases in \cite{LaSh2014}.
\begin{theorem}
If $\S$ has multiplication $\S(x,y) = \sum_{i=1}^r g_i(f_i(x)y)$, then its dual-transpose-dual $\S^{dtd}$ has multiplication 
\begin{equation}
\label{eqn:dtd}
\S^{dtd}(x,y) = \sum_{i=1}^r f_i(x)\hat{g}_i(y),
\end{equation}
where $\hat{g_i}$ denotes the adjoint of $g_i$ with respect to the symmetric bilinear form $(x,y) \mapsto \Tr(xy)$, with $\Tr$ denoting the field trace from $\Fqn$ to $\Fq$ where $\S$ is $n$-dimensional over $\Fq$ and $\Fq$ is contained in the centre of $\S$.
\end{theorem}

\begin{proof}
We have that $\S^d(x,y) = \sum_i g_i(f_i(y)x)$, and so the endomorphism of right multiplication by $y$ in $\S^d$ is given by  $R^d_y = \sum_i g_i f_i(y)$. Then $\widehat{R^d_y} = \sum_i \widehat{f_i(y)}\hat{g}$, and as $\widehat{f_i(y)} = f_i(y)$, we get that $\S^{dt}(x,y) = \widehat{R^d_y}(x) = \sum_i \hat{g}_i(x) f_i(y)$. Taking the dual of this gives the result.
\end{proof}

\section{BEL-rank and spread sets}
\label{sec:spreadsets}

Let $E$ denote the ring of $\Fq$-endomorphisms of $\Fqn$, i.e. $E := \End_{\Fq}(\Fqn)$. Consider a semifield $\S$ of order $\Fqn$ with centre containing $\Fq$. Then the set of maps of right-multiplication $R(\S) := \{R_y:x \mapsto \S(x,y)\mid y \in \Fqn\}$ define an $\Fq$-subspace of $E$ in which every nonzero element is invertible. This is known as a {\it semifield spread set}, see \cite{LaPo2011}. Similarly we may use the maps of left-multiplication to define $L(\S) := \{L_y:x \mapsto \S(y,x)\mid y \in \Fqn\}$, and then $L(\S) = R(\S^d)$. If we work in the projective space $\PG(E,q)$, which is isomorphic to $\PG(n^2-1,q)$, then the image of $R(\S)$ is known as a {\it geometric spread set}. 

The space $\PG(E,q)$ has two canonical Desarguesian spreads, see for example \cite[Theorem 1.6.4]{LavPhd}:
\begin{align*}
\D_r &:= \{ \langle f  \alpha :\alpha \in \Fqn\rangle:f \in E^{\times}\}\\
\D_l &:= \{ \langle \alpha f :\alpha \in \Fqn\rangle:f \in E^{\times}\}.
\end{align*}
Here we identify the field element $\alpha$ with the map $x \mapsto \alpha x$. Note that juxtaposition denotes composition, i.e. $(f \alpha) (x) = f(\alpha x)$. These spreads have been used to define some very useful geometric invariants for finite semifields, in particular in the case of semifields two-dimensional over a nucleus. 
See for example \cite{CAPOTR2004}, \cite{MaPoTr2007}.

For a spread $\D$, we will write $B_{\D}(f)$ for the unique element of $\D$ containing $f$. For a subset $U$ of $\PG(E,q)$, we denote $B_{\D}(U) = \{B_{\D}(f):f \in U\}$.

The {\it Segre variety} $S_{n,n}(q)$ in $\PG(E,q)$ is the image of the Segre map,
and corresponds to the points defined by rank one endomorphisms. The points of the Segre variety can be partitioned into maximum subspaces in two ways. We denote these two families of maximal subspaces by $\cF_l$ and $\cF_r$, where $\cF_i$ is contained in $\D_i$ for $i \in \{r,l\}$. We have the following lemma.

\begin{lemma}
\label{lem:goodspread}
Suppose $\D$ is a Desarguesian spread of $\PG(E,q)$ containing $\cF_r$ (resp. $\cF_l$). Then there exists an invertible  $X \in E$ such that $\D$ is of the form $\D_{r,X}$ (resp. $D_{l,X})$, where
\begin{align*}
\D_{r,X} &:= \{ \langle f  \alpha X :\alpha \in \Fqn\rangle:f \in E^{\times}\}\\
\D_{l,X} &:= \{ \langle X\alpha f :\alpha \in \Fqn\rangle:f \in E^{\times}\}.
\end{align*}
\end{lemma} 

\begin{proof}
Suppose $\D$ contains $\cF_r$. As all Desarguesian spreads in $\PG(E,q)$ are equivalent under the action of $\PGL(n^2,q)$, we have that $\D^{\phi}=\D_r$ for some $\phi \in \PGL(n^2,q)$. Then $\cF_r^{\phi}$ is one of the families of maximal spaces of a Segre variety $S_{n,n}(q)$ contained in $\D_r$. If we identify the elements of $\D_r$ with the points of $\PG(n-1,q^n)$, then any Segre variety in $\D_r$ corresponds to the points of a subgeometry $\PG(n-1,q)$. As all such subgeometries are projectively equivalent, all Segre varieties $S_{n,n}(q)$ contained in $\D_r$ are equivalent under the stabiliser of $\D_r$ in $\PGL(n^2,q)$. Hence there exists some  $\psi \in \PGL(n^2,q)$ stabilising $\D_r$ such that  $(\cF_r^{\phi})^{\psi}=\cF_r$. Then $\D^{\phi\psi} = \D_r$, and as $\phi \psi$ fixes $\cF_r$, it must be of the form $f \mapsto YfX^{-1}$ for some invertible $X,Y$. Hence $\D = Y^{-1}\D_r X = \D_{r,X}$, as claimed. The proof for Desarguesian spreads containing $\cF_l$ is similar.
\end{proof}

Now if $\cS$ has BEL-rank $r$, then $\cS$ is equivalent to $\cS(\S)$ for some $\S$ defined by $\S(x,y) = \sum_{i=1}^r g_i(f_i(x)y)$. But then
\[
L(\S) = \left\{ \sum_{i=1}^r g_i f_i(x) : x \in \Fqn\right\} \leq \langle B_{\D_r}(g_1),\ldots,B_{\D_r}(g_r)\rangle,
\]
i.e. $L(\S)$ is spanned by $r$ elements of $\D_r$. Conversely, if $L(\S)$ is spanned by $r$ elements of $\D_r$, then $\cS(\S)$ has BEL-rank at most $r$. Hence we have the following.
\begin{theorem}
Let $\S$ be a semifield. Then 
\begin{align*}
\brk(\S) &= \frac{1}{n}\min\{ \dim \langle B_{\D_r}(L(\S')) \rangle : \S'\simeq \S\}\\
	  &= \frac{1}{n}\min\{ \dim \langle B_{\D}(L(\S)) \rangle : \D \textrm{ Desarguesian}, \cF_r \subset \D\}.
\end{align*}
\end{theorem}
\begin{proof}
The first equality follows immediately from Definition \ref{def:brk}, Theorem \ref{thm:Sfg}, and the above discussion. For the second, suppose $\S$ is isotopic to $\S'$. Then by \cite[Theorem 18]{Lavrauw2008}
$L(\S') = XL(\S)^{\sigma}Y$ for some invertible $X,Y \in E$, and some $\sigma \in \Aut(\Fq)$. Then $\dim \langle B_{\D_r}(L(\S')) \rangle = \dim \langle B_{\D_{r,Y}}(L(\S)) \rangle$, and the result follows by Lemma \ref{lem:goodspread}.
\end{proof}
One could equally define an invariant for a semifield $\S$ by replacing $\cF_r$ with $\cF_l$, or by replacing $L(\S)$ with $R(\S)$. The following result shows that these in fact all correspond to the BEL-rank of semifields whose isotopism class belong to the Knuth-orbit of $\S$.

\begin{theorem}\label{thm:3belrks} The BEL-rank of the images of a semifield $\S$ under the Knuth actions are as follows.
\begin{align*}
\brk(\S) &=\brk(\S^t)= \frac{1}{n}\min\{ \dim \langle B_{\D}(L(\S)) \rangle : \D \textrm{ Desarguesian}, \cF_r \subset \D\},\\
\brk(\S^{dtd}) &=\brk(\S^{td})= \frac{1}{n}\min\{ \dim \langle B_{\D}(L(\S)) \rangle : \D \textrm{ Desarguesian}, \cF_l \subset \D\},\\
\brk(\S^d) &=\brk(\S^{dt})= \frac{1}{n}\min\{ \dim \langle B_{\D}(R(\S)) \rangle : \D \textrm{ Desarguesian}, \cF_r \subset \D\}.
\end{align*}
Equivalently,
\begin{align*}
\brk(\S) &=\brk(\S^t)= \frac{1}{n}\min\{ \dim \langle B_{\D_r}(L(\S')) \rangle : \S'\simeq \S\},\\
\brk(\S^{dtd}) &=\brk(\S^{td})= \frac{1}{n}\min\{ \dim \langle B_{\D_l}(L(\S')) \rangle : \S'\simeq \S\},\\
\brk(\S^d) &=\brk(\S^{dt})= \frac{1}{n}\min\{ \dim \langle B_{\D_r}(R(\S)) \rangle : \S'\simeq \S\}.
\end{align*}
\end{theorem}

Hence we can define three invariants for a semifield, all of which can be calculated from the left- and right-spread sets of $\S$. We compare this with the spread set interpretation of the nuclei of a semifield, as in \cite[Theorem 2.2]{MaPo2012}.
\[
b(\S) := (\brk(\S),\brk(\S^{d}),\brk(\S^{dt})).
\]

We now give relations between these invariants and the nuclei. The result essentially follows from the proof of  \cite[Theorem 4.1]{BEL2007}.
\begin{theorem}
Suppose $\S$ has dimension $m$ over its middle nucleus, and dimension $r$ over its right nucleus. Then
\[
\brk(\S) \leq \mathrm{min}\{m,r\}.
\]
\end{theorem}

\begin{proof}
If $\S$ has dimension $r$ over its right nucleus, then there exists a semifield $\S'$ isotopic to $\S$ such that
\[
\S'(x,y) = \sum_{i=0}^{r-1} c_i(x)y^{q^{\frac{ni}r}}.
\]
Taking $f_i(x) = c_i(x)^{q^{\frac{-nj}r}}$ and $g_i(x) = x^{q^{\frac{nj}r}}$, by Theorem \ref{thm:Sfg} we get that $\brk(\S)\leq r$. As $\brk(\S) = \brk(\S^t)$, and the dimension of $\S^t$ over its right nucleus is equal to the dimension of $\S$ over its middle nucleus, we get that  $\brk(\S)\leq m$.
\end{proof}

\begin{remark}
In Section \ref{sec:comp} we will see some computational examples where $\brk(\S)\ne \brk(\S^d)$, and so the definition of $b(\S)$ is non-trivial. We will also see examples where $\brk(\S)>l$, where $l$ is the dimension of $\S$ over its left nucleus.
\end{remark}

\section{The matrix rank of a semifield}
\label{sec:mrk}
Let $\SSS$ be an $n$-dimensional algebra over $\Fq$. Recall that there exist unique elements $c_{ij} \in \Fqn$ such that
\[
\SSS(x,y) = \sum_{i,j=0}^{n-1} c_{ij} x^{q^i}y^{q^j}.
\]

We define $M(\SSS)$ to be the matrix in $M_n(\Fqn)$ whose $(i,j)$-entry is $c_{ij}$. Given $x \in \Fqn$, define the $1 \times n$ matrix 
\begin{eqnarray}
\label{eqn:xvec}
\underline{x} = \npmatrix{x&x^q&\ldots&x^{q^{n-1}}}.
\end{eqnarray}
Then
\begin{eqnarray}
\label{eqn:matalg}
\SSS(x,y) = \underline{x} M(\SSS) \underline{y}^T.
\end{eqnarray}

\begin{definition}\label{def:mrk}
The \emph{matrix rank} of a finite algebra $\SSS$ is defined as
\[
\mrk(\S) := \rank(M(\S)).
\]
The matrix rank of an isotopy class $[\S]$ is defined as 
\[
\mrk([\S]) := \mathrm{min}\{ \mrk(\mathbb{T}) \mid \mathbb{T} \in [\S]\}
\]
\end{definition}

Suppose $\rank(M(\SSS)) = r$. Then there exist $1 \times n$ matrices $u_i,v_i$ such that 
\[
M(\SSS) = \sum_{i=1}^r u_i ^T v_i.
\]
Given a $1 \times n$ matrix $\npmatrix{w_0&w_1&\ldots&w_{n-1}}$ with entries in $\Fqn$, we define the linear map, also denoted by $w$, by
\[
w(x) = \sum_{j=0}^{n-1} w_j x^{q^j} = \underline{x} w^T.
\]
Hence 
\begin{align}
\nonumber
\SSS(x,y) &= \sum_{i=1}^r  \underline{x} u_i^T v_i \underline{y}^T\\
\label{eqn:mrk1}
&=  \sum_{i=1}^r u_i(x) v_i(y).
\end{align}

We also recall the definition of the autocirculant matrix $A_w$,
\[
A_w = \npmatrix{w_0 & w_1 &\cdots&w_{n-1}\\w_{n-1}^q&w_0^q&\cdots&w_{n-1}^q\\\vdots&\vdots&\ddots&\vdots\\ w_1^{q^{n-1}}&w_2^{q^{n-1}}&\cdots&w_0^{q^{n-1}}},
\]
and note that $A_w \underline{x}^T = \underline{w(x)}^T$.

\begin{theorem}
Suppose $\S$ and $\S'$ are strongly isotopic. Then $\mrk(\S) = \mrk(\S')$.
\end{theorem}

\begin{proof}

Suppose $\mrk(\S) = r$. Then $M(\S) = \sum_{i=1}^r u_i^T v_i$, for some $u_i,v_i$, and by (\ref{eqn:mrk1}) $\S(x,y) = \sum_{i=1}^r  \underline{x} u_i^T v_i \underline{y}^T=  \sum_{i=1}^r u_i(x) v_i(y)$. Then $\S'(x,y) = \S(F(x),G(y)) = \sum_{i=1}^r u_i (F(x)) v_i (G(y)) = \sum_{i=1}^r  \underline{F(x)} u_i^T v_i \underline{G(y)}^T = \sum_{i=1}^r  \underline{x} A_F^T u_i^T v_i A_G \underline{y}^T$ , and so $M(\S') = \sum_{i=1}^r  (u_i A_F)^T (v_i A_G)$, implying $\rank(M(\S')) \leq \rank(M(\S))$. Repeating the argument interchanging the roles of $\S$ and $\S'$, we get that $\rank(M(\S')) = \rank(M(\S))$, completing the proof.
\end{proof}

Hence to calculate $\mrk([\S])$, it suffices to calculate $\mrk(\S^{(I,I,H)})$ for all $H \in \GL(n,q)$. In fact, it also holds that $\mrk(\S^{(I,I,H_{a,i})}) = \mrk(\S)$ for all $a \in \Fqn$, $i \in \{0,\ldots,n-1\}$, where $H_{a,i}(x) := ax^{q^i}$. Hence we have the following.
\begin{lemma}
Let $\S$ be a semifield. Then
\[
\mrk([\S]) = \min\{\mrk(\S^{(I,I,H)}):H(x) = x + \sum_{j=1}^{n-1}h_j x^{q^j}, H \textrm{invertible}\}.
\]
\end{lemma}
 This greatly reduces the length of the computer calculations in Section \ref{sec:comp}.

The following is straightforward.
\begin{theorem}
\label{thm:field}
Let $\S$ be a semifield of order $q^n$ with centre containing $\Fq$. Then $\mrk([\S]) = 1$ if and only if $\S$ is isotopic to the field $\Fqn$.
\end{theorem}
\begin{proof}
Let $\S'$ be a semifield isotopic to $\S$ such that $\rank(M(\S'))=1$. Then by (\ref{eqn:mrk1}), there exist $\Fq$-linear maps $u,v$ such that $\S'(x,y) = u(x)v(y)$. As $\S'$ is a semifield, we must have that $u$ and $v$ are invertible. Hence $\S'$ is isotopic to $\Fqn$ via the  triple $(u,v,I)$, and so $\S$ is isotopic to $\Fqn$, as claimed.
\end{proof}

Now we calculate the matrix rank of Albert's generalized twisted fields, which are defined as follows. Let $\alpha, \beta \in \Aut(\Fqn:\Fq)$, and $c \in \Fqn$ such that $c \notin \{x^{\alpha-1}y^{\beta-1}:x,y \in \Fqn\}$. Then the multiplication on
$\Fqn$
\[
x\circ y = xy - cx^{\alpha}y^{\beta}
\]
defines a presemifield of order $q^n$. We will denote this presemifield by $\mathrm{GTF}_{c,\alpha,\beta}$.
If $x^{\alpha} = x^{q^k}$ and $y^{\beta} = y^{q^m}$, we see that
\[
c_{ij} = \left\{\begin{array}{ll}
1& \textrm{if $i=j=0$},\\
-c & \textrm{if $i=k,j=m$},\\
0 & \textrm{otherwise}.
\end{array}\right.
\]
Clearly $\rank(M(\GTF_{c,\alpha,\beta})) = 2$, and hence taking into account Theorem \ref{thm:field}, $\mrk([\GTF_{c,\alpha,\beta}])=2$.

\section{BEL-rank and matrix rank}
\label{sec:bmrk}
We now prove an important theorem, which allows us to efficiently calculate the BEL-rank of a semifield spread.
\begin{theorem}
\label{thm:brkmrk}
Let $\S$ be a semifield. Then
\[
\brk(\S) = \mrk([\S^{dtd}]).
\]
\end{theorem}
\begin{proof}
Suppose that $\brk(\S) = r$, and $\mrk([\S^{dtd}]) = k$. Then there exist maps $f_1,\ldots,f_r,g_1,\ldots,g_r$ such that $\S'(x,y) := \sum_{i=1}^r g_i(f_i(x)y)$, and $[\S'] = [\S]$. Then by (\ref{eqn:dtd}) $M(\S'^{dtd}) = \sum_{i=1}^r f_i^T \hat{g}_i$, and so $\rank(M(\S'^{dtd})) \leq r$, implying $\mrk([\S^{dtd}]) = k \leq r = \brk(\S)$.
\medskip

Let $\S'' \in [\S^{dtd}]$ be such that $\rank(M(\S'')) = k$. Then there exist row vectors $u_1,\ldots,u_k,v_1,\ldots,v_k$ such that $M(\S'') = \sum_{i=1}^k u_i^T v_i$. Then again by (\ref{eqn:dtd}) $\S''^{dtd}(x,y) = \sum_{i=1}^k \hat{v}_i (u_i(x)y)$, and so, by Theorem \ref{thm:Sfg}, there exists a BEL-configuration $(\D,U,W)$ in $V(kn,q)$ such that $\cS(\D,U,W) \simeq \cS(\S''^{dtd}) \simeq \cS(\S)$, and so $\brk(\S) = r \leq k$. Hence $r=k$, completing the proof.
\end{proof}

\section{BEL-rank and tensor rank}
\label{sec:trk}

Another isotopism invariant for semifields, the {\it tensor rank}, was defined in \cite{Lavrauw2011}, based on the interpretation of a semifield of order $q^n$ with centre containing $\Fq$ as a tensor in the three-fold tensor product space $V_1\otimes V_2\otimes V_3$, where each $V_i \simeq V(n,q)$. In this section we will compare this with the BEL-rank, by viewing them both as ranks with respect to certain sets in the same projective space.

Let $V = V(n,q)$ be the $n$-dimensional vector space over $\Fq$, and let $V^{\vee}$ be the dual space of $V$. If we identify the elements of $V$ with the set $\{\underline{x}:x \in \Fqn\}$ (where $\underline{x}$ is as defined in (\ref{eqn:xvec}), then we can take $V^{\vee} = \{\underline{a}^{\vee}:a \in \Fqn\} $, where $\underline{a}^{\vee}$ is defined by
\[
\underline{a}^{\vee}(\underline{x}) = \underline{a} ~\underline{x}^T = \Tr(ax).
\]

Now $n$-dimensional algebras over $\Fq$ are in one-to-one correspondence with tensors in the space $V^{\vee} \otimes V^{\vee} \otimes V$. See \cite{Lavrauw2011} for details regarding tensors and semifields. 

A {\it fundamental tensor} is a tensor of the form $T = \underline{a}^{\vee} \otimes \underline{b}^{\vee} \otimes \underline{c}$. The tensor product space is generated by the fundamental tensors. The {\it tensor rank} of a tensor $T$ is the minimum number of fundamental tensors needed to span a subspace containing $T$.

The tensor rank of a semifield $\S$ was defined in \cite{Lavrauw2011} as the tensor rank of a corresponding tensor. We now compare this to the BEL-rank, by viewing both as ranks with respect to different sets in the same space. 

A fundamental tensor $T = \underline{a}^{\vee} \otimes \underline{b}^{\vee} \otimes \underline{c}$ defines an algebra with multiplication $T(\underline{x},\underline{y}) = \underline{a}^{\vee}(\underline{x})\underline{b}^{\vee}(\underline{y})\underline{c}$. Abusing notation, this gives a multiplication on $\Fqn$ defined by
\[
T(x,y) = \Tr(ax)\Tr(by)c.
\]
As before we may define the matrix $M(T)$. We immediately see that $M = c \underline{a}^T \underline{b}$,
i.e. $M_{ij} = a^{q^i}b^{q^j}c$.

This can be extended linearly to obtain a matrix from an algebra corresponding to an arbitrary tensor, namely by defining $M(\sum_i T_i) := \sum_iM(T_i)$, where the $T_i$'s are tensors of rank 1. We will use $T$ to denote both a tensor and its corresponding algebra. This induces a map $T\mapsto M(T)$ from the space $\Fqn^{\vee} \otimes \Fqn^{\vee} \otimes \Fqn$ into $\PG(M_n(\Fqn),\Fq) \simeq \PG(n^3-1,q)$. Denote by $\Omega_0$ the image of the set of fundamental tensors under this map. 

If $\Omega, \Omega'$ are arbitrary sets in a vector (or projective) space, we define as usual the {\it $\Omega$-rank of $\Omega'$} as the minimum number of elements of $\Omega$ needed to span a space containing all elements of $\Omega'$. We denote it by $\rk_{\Omega}(\Omega')$. If $\omega$ is a vector (or projective point), we define the {\it $\Omega$-rank of $\omega$} to be the $\Omega$-rank of the singleton set $\{\omega\}$, and denote it by $\rk_{\Omega}(\omega)$.

Then the tensor rank of a tensor $T$ is precisely the $\Omega_0$-rank of the point $\langle M(T)\rangle$. This leads to the following relation between tensor rank and matrix rank.

\begin{theorem} For any semifield $\S$, it holds that $\trk(\S) \geq \mrk(\S)$.
\end{theorem}

\begin{proof}
Denote by $\Omega_1$ the subset of $\PG(n^3-1,q)$ induced by the set of rank one matrices in $M_n(\Fqn)$. By Definition \ref{def:mrk}, the matrix rank of $\S$ is precisely the $\Omega_1$-rank of $M(\S)$. Clearly $\Omega_0 \subset \Omega_1$, and hence the result follows.
\end{proof}

Since the tensor rank is invariant under isotopy \cite{Lavrauw2011}, the tensor rank of an isotopy class is well defined, and we immediately get the following.
\begin{corollary}For any semifield $\S$, it holds that $\trk(\S) \geq \mrk([\S])$.
\end{corollary}

Given a matrix $A \in M_n(\Fqn)$, we can define an algebra $S_A$ with multiplication $S_A(x,y) = \underline{x}S\underline{y}^T$, as in (\ref{eqn:matalg}). We can then define the isotopy class $[A]$ as the set of matrices in $M_n(\Fqn)$ defining algebras isotopic to $S_A$.

\begin{theorem}
If $T \in V^{\vee} \otimes V^{\vee} \otimes V$, then
\[
\trk(T) \leq \mathrm{max}\{\mrk([A]):A \in M_n(\Fq)\}.\rk_{\Omega_0}(\Omega_1).
\]
\end{theorem}
\begin{proof}
Let $m = \mathrm{max}\{\mrk([A]):A \in M_n(\Fq)\}$ and $k = \rk_{\Omega_0}(\Omega_1)$. Then $T$ is isotopic to an algebra $T'$ with $\rank(M(T')) \leq m$. Then $M(T')$ can be written as a sum of at most $m$ elements of $\Omega_1$, and each element of $\Omega_1$ can be written as a sum of at most $k$ elements of $\Omega_0$, and hence $\trk(T) = \trk(T') \leq mk$, proving the claim.
\end{proof}

This bound is unlikely to be particularly effective. However over finite fields there are no existing good upper bounds on the tensor rank.

\section{Computational results}
\label{sec:comp}
Using the computer classifications from \cite{Demp2008}, \cite{Rua2009},  \cite{Rua2011a}, \cite{Rua2011b}, we calculated the BEL-rank of all known semifields of various orders. We tabulate the results here, and point out some interesting features. Recall that all semifields three dimensional over their centre are fields or generalised twisted fields \cite{Menichetti1977}, and so their BEL-ranks are known. The largest calculation is for semifields of order 64, incidentally giving the most interesting results.

\subsection{Semifields of small order}

Here we use computer classifications from \cite{Demp2008}, \cite{Rua2009},  \cite{Rua2011a}, \cite{Rua2011b} to calculate the BEL-rank of all semifields of orders $2^4,2^5,3^4,3^5,5^4$, as well as all semifield of order $4^4$ with centre containing $\FF_4$. When the calculation showed that the BEL-rank is the same for all elements of each Knuth orbit, we tabulate the number of Knuth orbits of each BEL-rank, specified by $\# K$. Otherwise we tabulate the number of isotopy classes of each BEL-rank, specified by $\# I$.

$q=2$, $n=4$:

\begin{center}
\begin{tabular}{|c|c|c|c|c|}
\hline
brk&1&2&3&4\\
\hline
\# K&1&1&1&None\\
\hline
\end{tabular}
\end{center}
The semifield of BEL-rank two here is two-dimensional over a nucleus.

$q=2$, $n=5$:

\begin{center}
\begin{tabular}{|c|c|c|c|c|c|}
\hline
brk&1&2&3&4&5\\
\hline
\# K&1&None&1&1&None\\
\hline
\end{tabular}
\end{center}

There are no semifields of BEL-rank two. The BEL-rank is the same for each element of any given Knuth orbit.

$q=3$, $n=4$:

\begin{center}
\begin{tabular}{|c|c|c|c|c|}
\hline
brk&1&2&3&4\\
\hline
\# K&1&5&6&None\\
\hline
\end{tabular}
\end{center}
All semifields of BEL-rank two here are two-dimensional over a nucleus. The BEL-rank is the same for each element of any given Knuth orbit.

$q=3$, $n=5$: 
\begin{center}
\begin{tabular}{|c|c|c|c|c|c|}
\hline
brk&1&2&3&4&5\\
\hline
\# I&1&6&13&3&None\\
\hline
\end{tabular}
\end{center}
All semifields here of BEL-rank two here are generalised twisted fields. There is one Knuth orbit where the BEL-rank is not the same for each element: this has $b(\SSS) = (4,3,3)$.

$q=4$, $n=4$, centre containing $\FF_4$: 

\begin{center}
\begin{tabular}{|c|c|c|c|c|}
\hline
brk&1&2&3&4\\
\hline
\# K&1&5&22&None\\
\hline
\end{tabular}
\end{center}
All semifields of BEL-rank two here are two-dimensional over a nucleus. The BEL-rank is the same for each element of any given Knuth orbit.

$q=5$, $n=4$: 

\begin{center}
\begin{tabular}{|c|c|c|c|c|}
\hline
brk&1&2&3&4\\
\hline
\#I&1&32&78&None\\
\hline
\end{tabular}
\end{center}
There is one Knuth orbit where the BEL-rank is not the same for each element: this has $b(\SSS) = (2,3,3)$. Each of these semifields has trivial nuclei and is not equivalent to a generalised twisted field, and so we have two more examples of new semifields of BEL-rank two. All other semifields of BEL-rank two of this order are two-dimensional over a nucleus.

\subsection{Semifields of order $64$.} 
We calculated the BEL-ranks of each semifield of order $64$. The results are tabulated below, arranged by Knuth orbit as  in \cite{Rua2009}.

The calculation shows that there are four examples of semifields which can be constructed from BEL-configurations in $V(2n,q)$ which are neither twisted fields nor two-dimensional over a nucleus. These are the first known examples of such semifields. They are the semifields labelled XIV, XXIV, XXVI and XXVII in \cite{Rua2009}. We note that none are related to each other by the semifield operations $s$ and $e$ defined in \cite{BEL2007} and \cite{LaSh2014}. We now give a form for the multiplication arising from such a BEL-configuration for each of these. 

There is a semifield in Knuth orbit XIV which has multiplication defined by
\[
xy +(f(x)y)^{q^2},
\]
for some linear map $f$.

There is a semifield in each of the Knuth orbits XXIV, XXVI and XXVII which have multiplication defined by
\[
xy + \Tr_{q^6:q^2}(f(x)y),
\]
for some linear map $f$. We note that semifields with multiplication of the form $xy + \Tr(f(x)y)$, where $\Tr$ denotes the absolute trace function, have been studied in \cite{HoOzZh}, where it is shown that with certain restrictions on $q$ and $n$, all such semifields are isotopic to a field. The case where the trace function is taken to a non-prime field is left open. These examples, together with the examples of order $5^4$ noted in the preceding subsection, are the first non-trivial examples of semifields of this form.

We can see from the table below that there are many examples where $\brk(\S) \ne \brk(\S^d)$. Note also that semifields in orbit XIV are three-dimensional over one nucleus, and six-dimensional over the other two nuclei. The BEL-rank of a semifield in this Knuth-orbit is either $2$ or $4$. Hence we can have an example of a semifield such that $\brk(\S) > l$, where $l$ is the dimension of $\S$ over its left nucleus.

\[
\begin{array}{|c|ccc||c|ccc||c|ccc||c|ccc|}
\hline
&\S&\S^{d}&\S^{td}&&\S&\S^{d}&\S^{td}&&\S&\S^{d}&\S^{td}&&\S&\S^{d}&\S^{td}\\
\hline
1	&	1	&	1	&	1	&	21	&	4	&	4	&	4	&	41	&	3	&	3	&	3	&	61	&	3	&	3	&	3	\\
2	&	2	&	2	&	2	&	22	&	3	&	3	&	3	&	42	&	4	&	3	&	4	&	62	&	4	&	5	&	4	\\
3	&	3	&	3	&	3	&	23	&	3	&	3	&	3	&	43	&	3	&	3	&	3	&	63	&	3	&	3	&	3	\\
4	&	2	&	2	&	2	&	24	&	2	&	3	&	2	&	44	&	3	&	4	&	3	&	64	&	4	&	5	&	4	\\
5	&	2	&	2	&	2	&	25	&	4	&	4	&	4	&	45	&	3	&	3	&	3	&	65	&	4	&	4	&	4	\\
6	&	2	&	2	&	2	&	26	&	2	&	3	&	2	&	46	&	5	&	4	&	5	&	66	&	4	&	4	&	4	\\
7	&	2	&	2	&	2	&	27	&	2	&	3	&	2	&	47	&	4	&	4	&	4	&	67	&	3	&	3	&	3	\\
8	&	2	&	2	&	2	&	28	&	4	&	4	&	4	&	48	&	4	&	5	&	4	&	68	&	3	&	3	&	3	\\
9	&	2	&	2	&	2	&	29	&	4	&	4	&	4	&	49	&	5	&	5	&	5	&	69	&	3	&	3	&	3	\\
10	&	2	&	2	&	2	&	30	&	4	&	4	&	4	&	50	&	4	&	4	&	4	&	70	&	3	&	4	&	3	\\
11	&	3	&	3	&	3	&	31	&	4	&	4	&	4	&	51	&	3	&	3	&	3	&	71	&	4	&	4	&	4	\\
12	&	4	&	4	&	4	&	32	&	4	&	4	&	4	&	52	&	4	&	4	&	4	&	72	&	4	&	4	&	4	\\
13	&	4	&	4	&	4	&	33	&	5	&	3	&	5	&	53	&	3	&	3	&	3	&	73	&	4	&	4	&	4	\\
14	&	4	&	2	&	4	&	34	&	4	&	5	&	4	&	54	&	3	&	3	&	3	&	74	&	4	&	4	&	4	\\
15	&	4	&	4	&	4	&	35	&	5	&	4	&	5	&	55	&	4	&	3	&	4	&	75	&	4	&	4	&	4	\\
16	&	4	&	4	&	4	&	36	&	4	&	4	&	4	&	56	&	3	&	3	&	3	&	76	&	5	&	5	&	5	\\
17	&	3	&	3	&	3	&	37	&	4	&	5	&	4	&	57	&	4	&	4	&	4	&	77	&	4	&	4	&	4	\\
18	&	4	&	4	&	4	&	38	&	4	&	5	&	4	&	58	&	3	&	3	&	3	&	78	&	4	&	4	&	4	\\
19	&	3	&	3	&	3	&	39	&	5	&	5	&	5	&	59	&	4	&	3	&	4	&	79	&	4	&	4	&	4	\\
20	&	4	&	4	&	4	&	40	&	3	&	4	&	3	&	60	&	5	&	5	&	5	&	80	&	4	&	4	&	4	\\
\hline
\end{array}
\]

This data suggests the following open questions: 
\begin{center}
\begin{itemize}
\item
{\it Is the BEL-rank of any finite semifield of order $q^n$ with centre containing $\Fq$ is at most $n-1$?}
\item
{\it Can we construct or classify semifields which have BEL-rank two, but are neither generalised twisted fields nor two-dimensional over a nucleus?}
\end{itemize}
\end{center}

\bibliographystyle{plain}

\end{document}